\documentclass[11pt,twoside, leqno]{article}

\usepackage{amssymb}
\usepackage{amsmath}
\usepackage{amsthm}
\usepackage{colortbl,dcolumn}

\allowdisplaybreaks

\def\rr{{\mathbb R}}

\def\rn{{{\rr}^n}}

\def\zz{{\mathbb Z}}
\def\nn{{\mathbb N}}

\def\cx{{\mathcal X}}

\def\cm{{\mathcal M}}

\def\fz{\infty}
\def\az{\alpha}
\def\supp{{\mathop\mathrm{\,supp\,}}}

\def\loc{{\mathop\mathrm{\,loc\,}}}

\def\lz{\lambda}
\def\dz{\delta}

\def\ez{\epsilon}

\def\kz{\kappa}
\def\bz{\beta}

\def\sz{\sigma}

\def\wz{\widetilde}

\def\hs{\hspace{0.3cm}}

\def\ls{\lesssim}

\def\rbmo{{\mathop\mathrm{RBMO}}}
\def\rblo{{\mathop\mathrm{RBLO}}}

\def\einf{{\mathop{\mathrm{\,essinf\,}}}}

\def\gfz{\genfrac{}{}{0pt}{}}

\def\r{\right}
\def\lf{\left}

\newtheorem{thm}{Theorem}[section]
\newtheorem{lem}{Lemma}[section]
\newtheorem{prop}{Proposition}[section]
\newtheorem{rem}{Remark}[section]

\newtheorem{defn}{Definition}[section]

\numberwithin{equation}{section}

\begin{document}

\arraycolsep=1pt

\title{{\vspace{-5cm}\small\hfill\bf Front. Math. China, to appear}\\
\vspace{4.5cm}\Large\bf Spaces of Type BLO on Non-homogeneous Metric Measure
Spaces\footnotetext{\hspace{-0.35cm} 2000 {\it Mathematics Subject
Classification}. {Primary 42B35; Secondary 42B25, 42B20, 30L99.}
\endgraf{\it Key words and phrases.} upper doubling, geometrically doubling,
${\mathop\mathrm{RBLO}}(\mu)$, maximal operator, Calder\'on-Zygmund
maximal operator.
\endgraf
The first author is supported by the National Natural Science
Foundation (Grant No. 11001234) of China and Chinese Universities
Scientific Fund (Grant No. 2009JS34), and the second (corresponding)
author is supported by the National Natural Science Foundation
(Grant No. 10871025) of China and Program for Changjiang Scholars
and Innovative Research Team in University of China.}}
\author{Haibo Lin and Dachun Yang\,\footnote{Corresponding
author.}}
\date{ }
\maketitle

\begin{center}
\begin{minipage}{13.6cm}\small
{\noindent{\bf Abstract.} Let $({\mathcal X}, d, \mu)$ be a metric
measure space and satisfy the so-called upper doubling condition and
the geometrically doubling condition. In this paper, the authors
introduce the space ${\mathop\mathrm{RBLO}}(\mu)$ and prove that it
is a subset of the known space ${\mathop\mathrm{RBMO}}(\mu)$ in this
context. Moreover, the authors establish several useful
characterizations for the space ${\mathop\mathrm{RBLO}}(\mu)$. As an
application, the authors obtain the boundedness of the maximal
Calder\'on-Zygmund operators from $L^\infty(\mu)$ to
${\mathop\mathrm{RBLO}}(\mu)$.}
\end{minipage}
\end{center}

\section{Introduction\label{s1}}

\hskip\parindent Spaces of homogeneous type were introduced by
Coifman and Weiss \cite{cw71} as a general framework in which many
results from real and harmonic analysis on Euclidean spaces have
their natural extensions; see, for example, \cite{cw77,he,hk}.
Recall that a metric space $(\cx,\,d)$ equipped with a Borel measure
$\mu$ is called a {\it space of homogeneous type} if
$(\cx,\,d,\,\mu)$ satisfies the following {\it measure doubling
condition} that there exists a positive constant $C_\mu$ such that
for all balls $B\subset\cx$,
\begin{equation}\label{e1.1}
0<\mu(2B)\le C_\mu\mu(B),
\end{equation}
where and in what follows, a ball $B\equiv
B(c_B,\,r_B)\equiv\{x\in\cx:\,d(x,\,c_B)<r_B\}$, and for any ball
$B$ and $\rho\in(1,\,\fz)$, $\rho B\equiv B(c_B,\,\rho r_B)$. We
point out that in \cite{cw71} (see also \cite{cw77}), the metric $d$
appeared in the definition of spaces of homogeneous type was assumed
only to be a {\it quasi-metric}. However, in this paper, for
simplicity, we {\it always assume} that $d$ is a metric.

Meanwhile, many classical results concerning the theory of
Calder\'on-Zygmund operators and function spaces have been proved
still valid for non-doubling measures. In particular, let $\mu$ be a
non-negative Radon measure on $\rn$ which only satisfies the {\it
polynomial growth condition} that there exist positive constants $C$
and $\kz\in(0, n]$ such that for all $x\in\rn$ and $r\in(0,\,\fz)$,
$\mu(\{y\in\rn:\,|x-y|<r\})\le C r^\kz$. Such a measure does not
need to satisfy the doubling condition \eqref{e1.1}. The
$L^q(\mu)$-boundedness with $q\in(1,\,\fz)$ of Calder\'on-Zygmund
operators modeled on the Cauchy integral operator with respect to
such a measure, as well as the endpoint spaces of $L^q(\mu)$ scale
and the related mapping properties of operators, have been
successfully developed in this context. Some highlights of this
theory, are the introduction of the Hardy space $H^1$ and its dual
space, the regularized BMO space, by Tolsa \cite{t01}, the proof of
$Tb$ theorem by Nazarov, Treil and Volberg \cite{ntv}, and the
solution of the Painlev\'e problem by Tolsa \cite{t03}.

However, as pointed out by Hyt\"onen in \cite{h10}, notwithstanding
these impressive achievements, the Calder\'on -Zygmund theory with
non-doubling measures is not in all respects a generalization of the
corresponding theory of spaces of homogeneous type. The measures
satisfying the polynomial growth condition are different from, not
general than, the doubling measures.

To include the spaces of homogeneous type and Euclidean spaces with
a non-negative Radon measure satisfying a polynomial condition,
Hyt\"onen \cite{h10} introduced a new class of metric measure spaces
which satisfy the so-called upper doubling condition and the
geometrically doubling condition (see, respectively, Definitions
\ref{d1.1} and \ref{d1.2} below), and a notion of spaces of
regularized BMO. Later, Hyt\"onen and Martikainen \cite{hm} further
established a version of $Tb$ theorem in this setting.

Let $(\cx,\,d,\,\mu)$ be a metric space satisfying the upper
doubling condition and geometrically doubling condition. The main
purpose of this paper is to introduce the space
${\mathop\mathrm{RBLO}}(\mu)$ and prove that it is a subset of the
known space ${\mathop\mathrm{RBMO}}(\mu)$ in this context. Moreover,
we establish several useful characterizations, including the one in
terms of the natural maximal operator, for the space
${\mathop\mathrm{RBLO}}(\mu)$. As an application, we prove that if
the Calder\'on-Zygmund operator is bounded on $L^2(\mu)$, then the
corresponding maximal operator is bounded from $L^\fz(\mu)$ to
$\rblo(\mu)$.

Recently, an atomic Hardy space $H^1(\mu)$ in this setting was
introduced in \cite{hyy10} and it was proved in \cite{hyy10} that
$(H^1(\mu))^*=\rbmo(\mu)$. As an application, the boundedness of
Calder\'on-Zygmund operators from $H^1(\mu)$ to $L^1(\mu)$ was
obtained in \cite{hyy10}.

We now recall the upper doubling space in \cite{h10}.

\begin{defn}\label{d1.1}\rm
A metric measure space $(\cx,\,d,\,\mu)$ is called {\it upper
doubling} if $\mu$ is a Borel measure on $\cx$ and there exists a
dominating function $\lz:\,\cx\times(0,\,\fz)\rightarrow (0,\,\fz)$
and a positive constant $C_\lz$ such that for each $x\in\cx$, $r
\rightarrow \lz(x,\,r)$ is non-decreasing, and for all  $x\in\cx$
and $r\in(0,\,\fz)$,
\begin{equation}\label{e1.2}
\mu(B(x,\,r))\le\lz(x,\,r)\le C_\lz\lz(x,\,r/2).
\end{equation}
\end{defn}

In what follows, we write $\nu\equiv\log_2 C_\lz$ which can be
thought of as a dimension of the measure in some sense.

\begin{rem}\label{r1.1}\rm
\begin{enumerate}
\item[(i)] Obviously, a space of homogeneous type is a special case of the
upper doubling spaces, where one can take the dominating function
$\lz(x,\,r)\equiv\mu(B(x,\,r))$. Moreover, let $\mu$ be a
non-negative Radon measure on $\rn$ which only satisfies the
polynomial growth condition. By taking $\lz(x,\,r)\equiv Cr^\kz$, we
see that $(\rn,\,|\cdot|,\,\mu)$ is also an upper doubling measure
space.

\item[(ii)] It was proved in \cite{hyy10} that there exists a
dominating function $\wz\lz$ related to $\lz$ satisfying the
property that there exists a positive constant $C$ such that for all
$x,\,y\in\cx$ with $d(x,\,y)\le r$,
\begin{equation}\label{e1.3}
\wz\lz(x,\,r)\le C\wz\lz(y,\,r).
\end{equation}
Based on this, in this paper, we {\it always assume} that the
dominating function $\lz$ also satisfies \eqref{e1.3}.
\end{enumerate}
\end{rem}

Throughout the whole paper, we {\it also assume} that the underlying
metric space $(\cx,\,d)$ satisfies the following geometrically
doubling condition.

\begin{defn}\label{d1.2}\rm
A metric space $(\cx,\,d)$ is called {\it geometrically doubling} if
there exists some $N_0\in\nn\equiv\{1,\,2,\,\cdots\}$ such that for
any ball $B(x,\,r)\subset\cx$, there exists a finite ball covering
$\{B(x_i,\,r/2)\}_i$ of $B(x,\,r)$ such that the cardinality of this
covering is at most $N_0$.
\end{defn}

\begin{rem}\label{r1.2}\rm
Let $(\cx,\,d)$ be a metric space. In \cite[Lemma 2.3]{h10},
Hyt\"onen showed that the following statements are mutually
equivalent:
\begin{enumerate}
\item[(i)] $(\cx,\,d)$ is geometrically doubling.

\item[(ii)] For any $\ez\in(0,\,1)$ and any ball $B(x,\,r)\subset\cx$,
there exists a finite ball covering $\{B(x_i,\,\ez r)\}_i$ of
$B(x,\,r)$ such that the cardinality of this covering is at most
$N_0\ez^{-n}$, where and in what follows, $N_0$ is as in Definition
\ref{d1.2} and $n\equiv\log_2 N_0$.

\item[(iii)] For every $\ez\in(0,\,1)$, any ball $B(x,\,r)\subset\cx$
can contain at most $N_0\ez^{-n}$ centers $\{x_i\}_i$ of disjoint
balls with radius $\ez r$.

\item[(iv)] There exists $M\in\nn$ such that any ball
$B(x,\,r)\subset\cx$ can contain at most $M$ centers $\{x_i\}_i$ of
disjoint balls $\{B(x_i,\,r/4)\}^M_{i=1}$.

\end{enumerate}
\end{rem}

It is well known that spaces of homogeneous type are geometrically
doubling spaces; see \cite[p. 67]{cw71}. Conversely, if $(\cx,\,d)$
is a complete geometrically doubling metric spaces, then there
exists a Borel measure $\mu$ on $\cx$ such that $(\cx,\,d,\,\mu)$ is
a space of homogeneous type; see \cite{ls} and \cite{w}.

A metric measure space $(\cx,\,d,\,\mu)$ is called a {\it
non-homogeneous metric measure space} in this paper, if $\mu$ is
upper doubling and $(\cx,\,d)$ is geometrically doubling. The
motivation to develop a harmonic analysis on non-homogeneous metric
measure spaces can be found in \cite{h10} and also in
\cite{t04,cw77,cw71}.

The paper is organized as follows. Let $(\cx,\,d,\,\mu)$ be a
non-homogeneous metric measure space. In Section 2, we introduce the
space $\rblo(\mu)$ and obtain some useful properties of this space.
In Section 3, a characterization of $\rblo(\mu)$ in terms of the
natural maximal operator is established. In Section 4, we obtain the
boundedness of the maximal Calder\'on-Zygmund operators from
$L^\fz(\mu)$ to $\rblo(\mu)$.

Finally, we make some convention on symbols. Throughout the paper,
we denote by $C$, $\wz C$, $c$ and $\wz c$ {\it positive constants}
which are independent of the main parameters, but they may vary from
line to line. {\it Constant with subscript}, such as $C_1$, does not
change in different occurrences. If $f\le Cg$, we then write $f\ls
g$ or $g\gtrsim f$; and if $f\ls g\ls f$, we then write $f\sim g$.
Also, for any subset $E\subset\cx$, $\chi_E$ denotes the {\it
characteristic function} of $E$.

\section{The spaces $\rblo(\mu)$}\label{s2}

\hskip\parindent In this section, we introduce the space
$\rblo(\mu)$ and establish its several equivalent characterizations.

We begin with the coefficients $\dz(B,\,S)$ for all balls $B\subset
S$ which were introduced by Hyt\"onen in \cite{h10} as analogues of
Tolsa's numbers $K_{Q,\,R}$ from \cite{t01}; see also \cite{hyy10}.

\begin{defn}\label{d2.1}\rm
For all balls $B\subset S$, let
\begin{equation*}
\dz(B,\,S)\equiv\int_{(2S)\setminus
B}\frac{d\mu(x)}{\lz(c_B,\,d(x,\,c_B))}.
\end{equation*}
\end{defn}

To recall some useful properties of $\dz$ proved in \cite{hyy10}, we
first recall the notion of the $(\az,\,\bz)$-doubling property.
Though the measure condition \eqref{e1.1} is not assumed uniformly
for all balls in the non-homogeneous metric measure space
$(\cx,\,d,\,\mu)$, it was shown in \cite{h10} that there are still
many small and large balls that have the following
$(\az,\,\bz)$-doubling property.

\begin{defn}\label{d2.2}\rm
Let $\az,\,\bz\in(1,\,\fz)$. A ball $B(x,\,r)\subset\cx$ is called
{\it $(\az,\,\bz)$-doubling} if $\mu(\az B)\le\bz\mu(B)$.
\end{defn}

To be precise, it was proved in \cite{h10} that if a metric measure
space $(\cx,\,d,\,\mu)$ is upper doubling and
$\bz>C_\lz^{\log_2\az}=\az^\nu$, then for every ball
$B(x,\,r)\subset\cx$, there exists some
$j\in\zz_+\equiv\nn\cup\{0\}$ such that $\az^j B$ is
$(\az,\,\bz)$-doubling. Moreover, let $(\cx,\,d)$ be geometrically
doubling, $\bz>\az^n$ with $n\equiv\log N_0$ and $\mu$ a Borel
measure on $\cx$ which is finite on bounded sets. Hyt\"onen
\cite{h10} also showed that for $\mu$-almost every $x\in\cx$, there
exist arbitrarily small $(\az,\,\bz)$-doubling balls centered at
$x$. Furthermore, the radius of these balls may be chosen to be of
the form $\az^{-j}r$ for $j\in\nn$ and any preassigned number
$r\in(0,\,\fz)$. Throughout this paper, for any $\az\in(1,\,\fz)$
and ball $B$, ${\wz B}^\az$ denotes the {\it smallest
$(\az,\,\bz_\az)$-doubling ball of the form $\az^j B$ with
$j\in\zz_+$}, where
\begin{equation}\label{e2.1}
\bz_\az\equiv\max\lf\{\az^n,\,\az^\nu\r\}+30^n+30^\nu
=\az^{\max\{n,\,\nu\}}+30^n+30^\nu.
\end{equation}

The following useful properties of $\dz$ were proved in
\cite{hyy10}.

\begin{lem}\label{l2.1}
\begin{enumerate}
\item[\rm(i)] For all balls $B\subset R\subset S$,
$\dz(B,\,R)\le\dz(B,\,S)$.

\item[\rm(ii)] For any $\rho\in[1,\,\fz)$, there exists a positive
constant $C$, depending on $\rho$, such that for all balls $B\subset
S$ with $r_S\le\rho r_B$, $\dz(B,\,S)\le C$.

\item[\rm(iii)] For any $\az\in(1,\,\fz)$, there exists a positive
constant $\wz C$, depending on $\az$, such that for all balls $B$,
$\dz(B,\,{\wz B}^\az)\le {\wz C}$.

\item[\rm(iv)] There exists a positive constant $c$ such that for all
balls $B\subset R\subset S$, $\dz(B,\,S)\le\dz(B,\,R)+c\dz(R,\,S)$.
In particular, if $B$ and $R$ are concentric, then $c=1$.

\item[\rm(v)] There exists a positive constant $\wz c$ such that for all
balls $B\subset R\subset S$, $\dz(R,\,S)\le{\wz c}[1+\dz(B,\,S)]$;
moreover, if $B$ and $R$ are concentric, then
$\dz(R,\,S)\le\dz(B,\,S)$.
\end{enumerate}
\end{lem}

Inspired by the work of \cite{j05,hyy,h10}, we introduce the space
$\rblo(\mu)$ as follows. In what follows, $L^1_\loc(\mu)$ denotes
the {\it space of all $\mu$-locally integrable functions}.

\begin{defn}\label{d2.3}\rm
Let $\eta,\,\rho\in(1,\,\fz)$, and $\bz_\rho$ be as in \eqref{e2.1}.
A real-valued function $f\in L^1_{\loc}(\mu)$ is said to be in the
{\it space $\rblo(\mu)$} if there exists a non-negative constant $C$
such that for all balls $B$,
\begin{equation}\label{e2.2}
\frac{1}{\mu(\eta B)}\int_B\lf[f(y)-{\mathop\einf_{{\wz
B}^\rho}}f\r]\,d\mu(y)\le C,
\end{equation}
and that for all $(\rho,\,\bz_\rho)$-doubling balls $B\subset S$,
\begin{equation}\label{e2.3}
{\mathop\einf_B}f-{\mathop\einf_S}f\le C[1+\dz(B,\,S)].
\end{equation}
Moreover, the $\rblo(\mu)$ {\it norm} of $f$ is defined to be the
minimal constant $C$ as above and denoted by $\|f\|_{\rblo(\mu)}$.
\end{defn}

\begin{rem}\label{r2.1}\rm
\begin{enumerate}
\item[(i)] It is obvious that $L^\fz(\mu)\subset\rblo(\mu)$. Moreover,
if $f\in\rblo(\mu)$, then $f+C$ with any fixed $C\in\rr$ also
belongs to $\rblo(\mu)$ and $\|f+C\|_{\rblo(\mu)}
=\|f\|_{\rblo(\mu)}$. Based on this, in this paper, we identify $f$
with its {\it equivalent class} $\{f+C:\ C\in\rr\}$, namely, we
regard $\rblo(\mu)$ as the {\it quotient space} $\rblo(\mu)/\rr$.

\item[(ii)] The classical space ${\mathop\mathrm{BLO}}(\rn)$ is defined
by Coifman and Rochberg \cite{cr80}. Let $\mu$ be a non-negative
Radon measure on $\rn$ which only satisfies the polynomial growth
condition. In the setting of $(\rn,\,|\cdot|,\,\mu)$, the space
$\rblo(\mu)$ was first introduced by Jiang \cite{j05} and improved
by \cite{hyy}. Moreover, in this setting, the space $\rblo(\mu)$
defined as in Definition \ref{d2.3} is just the one introduced in
\cite{hyy}.

\item[(iii)] The definition of $\rblo(\mu)$ is independent of the
choice of the constants $\eta,\,\rho\in(1,\,\fz)$; see Propositions
\ref{p2.1} and \ref{p2.2} below.
\end{enumerate}
\end{rem}

Let $\eta\in(1,\,\fz)$. Suppose that for any given $f\in
L^1_\loc(\mu)$, there exist a non-negative constant $\wz C$ and a
real number $f_B$ for any ball $B$ such that for all balls $B$,
\begin{equation}\label{e2.4}
\frac{1}{\mu(\eta B)}\int_B\lf[f(y)-f_B\r]\,d\mu(y)\le \wz C,
\end{equation}
that for all balls $B\subset S$,
\begin{equation}\label{e2.5}
|f_B-f_S|\le{\wz C}\,[1+\dz(B,\,S)],
\end{equation}
and that for all balls $B$,
\begin{equation}\label{e2.6}
f_B\le{\mathop\einf_B}f.
\end{equation}
We then define the {\it norm} $\|f\|_{**,\,\eta}\equiv\inf\{\wz
C\}$, where the infimum is taken over all the  non-negative
constants $\wz C$ as above.

\begin{prop}\label{p2.1}
The norm $\|\cdot\|_{**,\,\eta}$ is independent of the choice of the
constant $\eta\in(1,\,\fz)$.
\end{prop}

\begin{proof}
Let $\rho>\eta>1$ be some fixed constants. Obviously,
$\|f\|_{**,\,\rho}\le\|f\|_{**,\,\eta}$. So we only have to show
that $\|f\|_{**,\,\eta}\ls\|f\|_{**,\,\rho}$.

For the norm $\|f\|_{**,\,\rho}$, there exists a fixed collection
$\{f_B\}_B$ of real numbers  satisfying \eqref{e2.4} through
\eqref{e2.6} with the constant $\wz C$ replace by
$\|f\|_{**,\,\rho}$. Fix $\ez\in(0,\,(\eta-1)/\rho)$ and consider a
fixed ball $B_0\equiv B(x_0,\,r)$. Then, by Remark \ref{r1.2}(ii),
there exists a family $\{B_i\equiv B(x_i,\,\ez r):\,x_i\in
B_0\}_{i\in I}$ of balls, which cover $B_0$, where $\sharp I\le
N_0\ez^{-n}$. Here and in what follows, for any set $I$, we use
$\sharp I$ to denote the {\it cardinality} of $I$. Moreover, $\rho
B_i=B(x_i,\,\ez\rho r)\subset B(x_0,\,\eta r)=\eta B_0$, since
$r+\ez\rho r<\eta r$. By this, \eqref{e2.5}, and (ii) and (iv) of
Lemma \ref{l2.1}, we have that
\begin{align*}
|f_{B_i}-f_{B_0}|&\le|f_{B_i}-f_{\eta B_0}|+|f_{\eta
B_0}-f_{B_0}|\\
&\le\|f\|_{**,\,\rho}[2+\dz(B_i,\,\eta
B_0)+\dz(B_0,\,\eta B_0)]\\
&\ls\|f\|_{**,\,\rho}[1+\dz(B_i,\,\rho B_i)+\dz(\rho B_i,\,\eta
B_0)]\ls\|f\|_{**,\,\rho}.
\end{align*}
Thus, by this estimate and $\rho B_i\subset\eta B_0$ again, we
obtain
\begin{align*}
\int_{B_0}|f(y)-f_{B_0}|\,d\mu(y)&\le\sum_{i\in
I}\int_{B_i}|f(y)-f_{B_0}|\,d\mu(y)\\
&\le\sum_{i\in
I}\lf\{\int_{B_i}|f(y)-f_{B_i}|\,d\mu(y)+\mu(B_i)|f_{B_i}-f_{B_0}|\r\}\\
&\ls\sum_{i\in I}\|f\|_{**,\,\rho}\mu(\rho B_i)\ls
\|f\|_{**,\,\rho}\mu(\eta B_0),
\end{align*}
which, together with \eqref{e2.6} and the fact that \eqref{e2.5}
holds with the constant $\wz C$ replaced by $\|f\|_{**,\,\rho}$,
yields that $\|f\|_{**,\,\eta}\ls\|f\|_{**,\,\rho}$. This finishes
the proof of Proposition \ref{p2.1}.
\end{proof}

Based on Proposition \ref{p2.1}, from now on, we write
$\|\cdot\|_{**}$ instead of $\|\cdot\|_{**,\,\eta}$.

\begin{prop}\label{p2.2}
Let $\eta,\,\rho\in(1,\,\fz)$, and $\bz_\rho$ be as in \eqref{e2.1}.
Then the norms $\|\cdot\|_{**}$ and $\|\cdot\|_{\rblo(\mu)}$ are
equivalent.
\end{prop}

\begin{proof}
Suppose that $f\in L^1_\loc(\mu)$. We first show that
\begin{equation}\label{e2.7}
\|f\|_{**}\ls\|f\|_{\rblo(\mu)}.
\end{equation}
For any ball $B$, let $f_B\equiv{\mathop\einf_{{\wz B}^\rho}}f$.
Then \eqref{e2.4} and \eqref{e2.6} hold with ${\wz
C}\equiv\|f\|_{\rblo(\mu)}$. For any two balls $B\subset S$, to show
\eqref{e2.5},  we consider two cases.

{\it Case} (i) $r_{{\wz S}^\rho}\ge r_{{\wz B}^\rho}$. In this case,
${\wz B}^\rho\subset2{\wz S}^\rho$. Let $S_0\equiv \wz{2{\wz
S}^\rho}^\rho$. It follows from Lemma \ref{l2.1} that $\dz({\wz
S}^\rho,\,S_0)\ls1$ and $\dz({\wz B}^\rho, S_0)\ls 1+\dz(B,\,S)$,
which together with \eqref{e2.3} shows that
\begin{align*}
|f_B-f_S|=\lf|{\mathop\einf_{{\wz B}^\rho}}f-{\mathop\einf_{{\wz
S}^\rho}}f\r|&\le\lf|{\mathop\einf_{{\wz
B}^\rho}}f-{\mathop\einf_{S_0}}f\r|+\lf|{\mathop\einf_{S_0}}f-{\mathop\einf_{{\wz
S}^\rho}}f\r|\\
&\le[2+\dz({\wz B}^\rho,\,S_0)+\dz({\wz
S}^\rho,\,S_0)]\|f\|_{\rblo(\mu)}\\
&\ls [1+\dz(B,\,S)]\|f\|_{\rblo(\mu)}.
\end{align*}

{\it Case} (ii) $r_{{\wz S}^\rho}< r_{{\wz B}^\rho}$. In this case,
${\wz S}^\rho\subset2{\wz B}^\rho$. Notice that $r_{{\wz S}^\rho}\ge
r_B$. Thus, there exists a unique $m\in\nn$ such that
$r_{\rho^{m-1}B}\le r_{{\wz S}^\rho}<r_{\rho^m B}$ and $r_{\rho^m
B}\le r_{{\wz B}^\rho}$, since $r_{{\wz S}^\rho}< r_{{\wz B}^\rho}$.
Therefore, ${\wz S}^\rho\subset 2\rho^m B\subset 2{\wz B}^\rho$. Set
$B_0\equiv{\wz{2{\wz B}^\rho}}^\rho$. Then another application of
Lemma \ref{l2.1} implies that $\dz({\wz B}^\rho,\,B_0)\ls1$ and
$$
\dz({\wz S}^\rho,\,B_0)\ls\dz({\wz S}^\rho,\,2\rho^m B)+\dz(2\rho^m
B,\,B_0)\ls1.
$$
An argument similar to Case (i) also establishes \eqref{e2.5} in
this case. Thus, \eqref{e2.5} always holds.

Now let us show the converse of \eqref{e2.7}. For $f\in
L^1_\loc(\mu)$, assume that there exists a sequence $\{f_B\}_B$ of
real numbers satisfying \eqref{e2.4} through \eqref{e2.6} with the
non-negative constant $\wz C$ replaced by $\|f\|_{**}$. For any ball
$B$, by \eqref{e2.5}, \eqref{e2.6} and Lemma \ref{l2.1},
$$
f_B-{\mathop\einf_{{\wz B}^\rho}} f=f_B-f_{{\wz B}^\rho}+f_{{\wz
B}^\rho}-{\mathop\einf_{{\wz B}^\rho}} f\le[1+\dz(B,\,{\wz
B}^\rho)]\|f\|_{**}\ls\|f\|_{**}.
$$
This together with \eqref{e2.4} yields that for any ball $B$,
\begin{align*}
&\frac{1}{\mu(\eta B)}\int_B\lf[f(y)-{\mathop\einf_{{\wz
B}^\rho}}f\r]\,d\mu(y)\\
&\hs=\frac{1}{\mu(\eta
B)}\int_B\lf[f(y)-f_B\r]\,d\mu(y)+\frac{\mu(B)}{\mu(\eta
B)}\lf[f_B-{\mathop\einf_{{\wz B}^\rho}} f\r]\ls\|f\|_{**}.
\end{align*}

On the other hand, for any $(\rho,\,\bz_\rho)$-doubling ball $B$,
since \eqref{e2.4} holds with $\rho$ by Proposition \ref{p2.1}, we
then have
$$
\frac{1}{\mu(B)}\int_B[f(y)-f_B]\,d\mu(y)\le\frac{\mu(\rho
B)}{\mu(B)}\|f\|_{**}\ls\|f\|_{**}.
$$
Then from \eqref{e2.5} and \eqref{e2.6}, it follows that for any two
$(\rho,\,\bz_\rho)$-doubling balls $B\subset S$,
\begin{align*}
{\mathop\einf_B}f-{\mathop\einf_S}f&\le{\mathop\einf_B}f-f_B+f_B-f_S\\
&\le\frac{1}{\mu(B)}\int_B[f(y)-f_B]\,d\mu(y)+[1+\dz(B,\,S)]\|f\|_{**}\\
&\ls[1+\dz(B,\,S)]\|f\|_{**}.
\end{align*}
This establishes the converse of \eqref{e2.7}, and hence finishes
the proof of Proposition \ref{p2.2}.
\end{proof}

\begin{rem}\label{r2.2}\rm
In \cite{h10}, the space $\rbmo(\mu)$ was defined in the following
way, namely, let $\eta\in(1,\,\fz)$, a function $f\in L^1(\mu)$ is
said to be in the {\it space $\rbmo(\mu)$} if there exists a
non-negative constant $C$ and a complex number $f_B$ for any ball
$B$ such that for all balls $B$,
$$
\frac{1}{\mu(\eta B)}\int_B|f(y)-f_B|\,d\mu(y)\le C
$$
and that for all balls $B\subset S$,
$$
|f_B-f_S|\le C[1+\dz(B,\,S)].
$$
Moreover, the {\it $\rbmo(\mu)$ norm} of $f$ is defined to be the
minimal constant $C$ as above and denoted by $\|f\|_{\rbmo(\mu)}$.
From \cite[Lemma 4.6]{h10}, Propositions \ref{p2.1} and \ref{p2.2},
it is easy to follow that $\rblo(\mu)\subset\rbmo(\mu)$.
\end{rem}

\begin{prop}\label{p2.3}
Let $\eta,\,\rho\in(1,\,\fz)$, and $\bz_\rho$ be as in \eqref{e2.1}.
For $f\in L^1_\loc(\mu)$, the following statements are equivalent:
\begin{enumerate}
\item[\rm(i)] $f\in\rblo(\mu)$.

\item[\rm(ii)] There exists a non-negative constant $C_1$ satisfying
\eqref{e2.3} and that for all $(\rho,\,\bz_\rho)$-doubling balls
$B$,
\begin{equation}\label{e2.8}
\frac{1}{\mu(B)}\int_B\lf[f(y)-{\mathop\einf_B} f\r]\,d\mu(y)\le
C_1.
\end{equation}

\item[\rm(iii)] There exists a non-negative constant $C_2$ satisfying
\eqref{e2.8} and that for all $(\rho,\,\bz_\rho)$-doubling balls
$B\subset S$,
\begin{equation}\label{e2.9}
m_B(f)-m_S(f)\le C_2[1+\dz(B,\,S)],
\end{equation}
where and in what follows, $m_B(f)$ denotes the mean of $f$ over
$B$, namely, $m_B(f)\equiv\frac{1}{\mu(B)}\int_B f(y)\,d\mu(y)$.
\end{enumerate}
Moreover, the minimal constants $C_1$ and $C_2$ as above are
equivalent to $\|f\|_{\rblo(\mu)}$.
\end{prop}

To prove Proposition \ref{p2.3}, we need the following lemma, which
is a simple corollary of \cite[Theorem 1.2]{he} and \cite[Lemma
2.5]{h10}; see also \cite[Lemma 2.2]{hyy10}.

\begin{lem}\label{l2.2}
Let $(\cx,\,d)$ be a geometrically doubling metric space. Then every
family $\cal F$ of balls of uniformly bounded diameter contains an
at most countable disjointed subfamily $\cal G$ such that
$\cup_{B\in{\cal F}} B\subset\cup_{B\in{\cal G}}5B$.
\end{lem}

\begin{proof}[Proof of Proposition \ref{p2.3}]
By Propositions \ref{p2.1} and \ref{p2.2}, it suffices to show
Proposition \ref{p2.3} with $\eta\equiv6/5$ and $\rho=6$. It is easy
to see that (i) implies (ii) automatically.

We now prove that (ii) implies (iii). From \eqref{e2.3} together
with \eqref{e2.8}, it follows that for any two
$(6,\,\bz_6)$-doubling balls $B\subset S$,
$$
m_B(f)-m_S(f)\le m_B(f)-{\mathop\einf_B} f+{\mathop\einf_B}
f-{\mathop\einf_S} f\ls C_1[1+\dz(B,\,S)],
$$
which implies (iii).

Finally, assuming that (iii) holds, we show $f\in\rblo(\mu)$ by
Definition \ref{d2.3}. If $B$ is a $(6,\,\bz_6)$-doubling ball, then
by \eqref{e2.8}, \eqref{e2.2} holds. Let $B$ be any ball which is
not $(6,\,\bz_6)$-doubling. For $\mu$-almost every $x\in B$, let
$B_x$ be the {\it biggest $(30, \bz_6)$-doubling ball with center
$x$ and radius $30^{-k}r_B$ for some $k\in\nn$}. Recall that such
ball exists by \cite[Lemma 3.3]{h10}. Moreover, $B_x$ and $5B_x$ are
also $(6, \bz_6)$-doubling balls. Since $B$ is not
$(6,\,\bz_6)$-doubling, then ${\wz B}^6$ has the radius at least
$6r_B$. From this, it follows that $B_x\subset(6/5)B\subset{\wz
B}^6$. Let $A_x$ be the {\it smallest $(30,\,\bz_6)$-doubling ball
of the form $30^kB_x$ for some $k\in\nn$}, which exists by
\cite[Lemma 3.2]{h10}. Then $r_{A_x}\ge r_B$. To verify
\eqref{e2.2}, we first claim that
\begin{equation}\label{e2.10}
{\mathop\einf_{B_x}} f-{\mathop\einf_{{\wz B}^6}} f\ls C_2.
\end{equation}
To show \eqref{e2.10}, we consider the following two cases.

{\it Case} (i) $r_{{\wz B}^6}\le r_{A_x}$. In this case, ${\wz
B}^6\subset 2A_x$. Notice that $B_x$ is also $(6, \bz_6)$-doubling.
From (iv), (ii) and (iii) of Lemma \ref{l2.1}, we deduce that
$\dz(B_x,\,\wz{2A_x}^6)\ls1$. This, combined with \eqref{e2.9} and
\eqref{e2.8}, yields that
\begin{align*}
{\mathop\einf_{B_x}} f-{\mathop\einf_{{\wz B}^6}} f &\le
m_{B_x}(f)-m_{\wz{2A_x}^6}(f)+m_{\wz{2A_x}^6}(f)-{\mathop\einf_{\wz{2A_x}^6}}
f\\
&\ls C_2\lf[1+\dz(B_x,\,\wz{2A_x}^6)\r]\ls C_2.
\end{align*}

{\it Case} (ii) $r_{{\wz B}^6}> r_{A_x}$. In this case, since
$r_{A_x}\ge r_B$, then $B\subset 2A_x\subset 3{\wz B}^6$. This,
together with \eqref{e2.9}, \eqref{e2.8}, the fact that $B_x$ is
also $(6, \bz_6)$-doubling, and Lemma \ref{l2.1}, we have that
\begin{align*}
&{\mathop\einf_{B_x}} f-{\mathop\einf_{{\wz B}^6}} f \\
&\hs\le m_{B_x}(f)-m_{\wz{3{\wz B}^6}^6}(f)+m_{\wz{3{\wz
B}^6}^6}(f)-{\mathop\einf_{\wz{3{\wz B}^6}^6}}
f\\
&\hs\ls C_2\lf[1+\dz(B_x,\,\wz{3{\wz B}^6}^6)\r]
\ls C_2\lf[1+\dz(B_x,\,2A_x)+\dz(2A_x,\,\wz{3{\wz B}^6}^6)\r]\\
&\hs\ls C_2\lf[1+\dz(B,\,\wz{3{\wz B}^6}^6)\r]\ls C_2.
\end{align*}
Thus, \eqref{e2.10} holds. That is, the claim is true.

Now, by Lemma \ref{l2.2}, there exists a countable disjoint
subfamily $\{B_i\}_i$ of $\{B_x\}_x$ such that for $\mu$-almost
every $x\in B$, $x\in\cup_i 5B_i$. Moreover, since for any $i$,
$B_i$ and $5B_i$ are $(6,\,\bz_6)$-doubling, by \eqref{e2.8} and
\eqref{e2.10}, we have
\begin{align}\label{e2.11}
&\int_B\lf[f(y)-{\mathop\einf_{{\wz B}^6}}f\r]\,d\mu(y)\\
&\hs\le\sum_i\int_{5B_i}\lf|f(y)-{\mathop\einf_{{\wz
B}^6}}f\r|\,d\mu(y)\nonumber\\
&\hs\le\sum_i\int_{5B_i}\lf[f(y)-{\mathop\einf_{5B_i}}f\r]\,d\mu(y)
+\sum_i\lf[{\mathop\einf_{5B_i}} f-{\mathop\einf_{{\wz B}^6}}
f\r]\mu(5B_i)\nonumber\\
&\hs\ls C_2\sum_i\mu(5B_i)+\sum_i\lf[{\mathop\einf_{B_i}}
f-{\mathop\einf_{{\wz B}^6}}
f\r]\mu(5B_i)\nonumber\\
&\hs\ls C_2\sum_i\mu(5B_i)\ls C_2\sum_i\mu(B_i)\ls
C_2\mu\lf(\frac{6}{5}B\r).\nonumber
\end{align}
On the other hand, from \eqref{e2.8} and \eqref{e2.9}, it follows
that for any two $(6,\,\bz_6)$-doubling balls $B\subset S$,
$$
{\mathop\einf_B}f-{\mathop\einf_S}f\le
m_B(f)-m_S(f)+m_S(f)-{\mathop\einf_S}f\ls C_2[1+\dz(B,\,S)].
$$
This together with \eqref{e2.11} shows that $f\in\rblo(\mu)$ and
$\|f\|_{\rblo(\mu)}\ls C_2$, which implies (i), and hence completes
the proof of Proposition \ref{p2.3}.
\end{proof}

\section{A characterization of $\rblo(\mu)$
in terms of the natural maximal operator}\label{s3}

\hskip\parindent In this section, we give a characterization of
$\rblo(\mu)$ in terms of the {\it natural maximal operator}. This
characterization in $\rn$ equipped with the $n$-dimensional Lebesgue
measure was obtained by Bennett \cite{b82}. In $\rn$ equipped with a
non-doubling measure with polynomial growth, this characterization
was first established by Jiang \cite{j05} and was improved in
\cite{hyy}.

We begin with the notion of the natural maximal operator, which is a
variant of the maximal operator introduced by Hyt\"onen in
\cite{h10}. In the non-doubling context, the natural maximal
operator was introduced by Jiang in \cite{j05}. For any $f\in
L^1_\loc(\mu)$ and $x\in\cx$, define
$$
\cm(f)(x)\equiv\sup_{\gfz{B\ni x}{B\,(6,\,\bz_6)-\rm
doubling}}\frac{1}{\mu(B)}\int_Bf(y)\,d\mu(y).
$$
Obviously, for all $x\in\cx$, $\cm(f)(x)\ls{\wz M}f(x)$, where the
maximal operator $\wz M$ is defined by setting, for all $x\in\cx$,
$$
{\wz M}(f)(x)\equiv\sup_{B\ni
x}\frac{1}{\mu(6B)}\int_B|f(y)|\,d\mu(y).
$$
By \cite[Proposition 3.5]{h10},  we know that $\wz M$ is of weak
type $(1,\,1)$ and bounded on $L^p(\mu)$ with $p\in(1,\,\fz]$. As a
consequence, $\cm$ is also of weak type $(1,\,1)$ and bounded on
$L^p(\mu)$ with $p\in(1,\,\fz]$.

\begin{lem}\label{l3.1}
$f\in\rblo(\mu)$ if and only if $\cm(f)-f\in L^\fz(\mu)$ and $f$
satisfies \eqref{e2.9}. Furthermore,
\begin{equation}\label{e3.1}
\|\cm(f)-f\|_{L^\fz(\mu)}\sim\|f\|_{\rblo(\mu)}.
\end{equation}
\end{lem}

\begin{proof}
By \cite[Corollary 3.6]{h10}, we know that for any $f\in
L^1_\loc(\mu)$ and $\mu$-almost every $x\in\cx$,
$$
f(x)=\lim_{\gfz{B\downarrow x}{B\,(6,\,\bz_6)-\rm
doubling}}\frac1{\mu(B)}\int_Bf(y)\,d\mu(y),
$$
where the limit is along the decreasing family of all
$(6,\,\bz_6)$-doubling balls containing $x$, ordered by set
inclusion.  Using this fact and following the proof of \cite[Lemma
1]{j05}, we can show Lemma \ref{l3.1}. We omit the details, which
completes the proof of Lemma \ref{l3.1}.
\end{proof}

\begin{thm}\label{t3.1}
Let $f\in\rbmo(\mu)$. Then $\cm(f)$ is either infinite everywhere or
finite almost everywhere, and in the later case, there exists a
positive constant $C$, independent of $f$, such that
$$\|\cm(f)\|_{\rblo(\mu)}\le C\|f\|_{\rbmo(\mu)}.$$
\end{thm}

From Lemma \ref{l3.1} and Theorem \ref{t3.1}, we immediately deduce
the following result. We omit the details.

\begin{thm}\label{t3.2}
A locally integrable function $f$ belongs to $\rblo(\mu)$ if and
only if there exist $h\in L^\fz(\mu)$ and $g\in\rbmo(\mu)$ with
$\cm(g)$ finite $\mu$-almost everywhere such that
\begin{equation}\label{e3.2}
f=\cm(g)+h.
\end{equation}
Furthermore,
$\|f\|_{\rblo(\mu)}\sim\inf(\|g\|_{\rbmo(\mu)}+\|h\|_{L^\fz(\mu)})$,
where the infimum is taken over all representations of $f$ as in
\eqref{e3.2}.
\end{thm}

To prove Theorem \ref{t3.1}, we need the following characterization
of $\rbmo(\mu)$.

\begin{lem}\label{l3.2}
Let $\eta,\,\rho\in(1,\,\fz)$, and $\bz_\rho$ be as in \eqref{e2.1}.
For $f\in L^1_\loc(\mu)$, the following statements are equivalent:
\begin{enumerate}
\item[\rm(i)] $f\in\rbmo(\mu)$.

\item[\rm(ii)] There exists a non-negative constant $C_3$ such that for all
$(\rho,\,\bz_\rho)$-doubling balls $B$,
\begin{equation}\label{e3.3}
\frac{1}{\mu(B)}\int_B\lf|f(y)-m_B(f)\r|\,d\mu(y)\le C_3,
\end{equation}
and that for all $(\rho,\,\bz_\rho)$-doubling balls $B\subset S$,
\begin{equation}\label{e3.4}
|m_B(f)-m_S(f)|\le C_3[1+\dz(B,\,S)].
\end{equation}

\item[\rm(iii)] There exists a non-negative constant $C_4$ satisfying \eqref{e3.4}
and that for all balls $B$,
\begin{equation}\label{e3.5}
\frac{1}{\mu(\eta B)}\int_B\lf|f(y)-m_{{\wz
B}^\rho}(f)\r|\,d\mu(y)\le C_4.
\end{equation}

\item[\rm(iv)] Let $p\in[1,\,\fz)$. There exists a non-negative constant
$C_5$ satisfying \eqref{e3.4} and that for all balls $B$,
\begin{equation}\label{e3.6}
\lf\{\frac{1}{\mu(\eta B)}\int_B\lf|f(y)-m_{{\wz
B}^\rho}(f)\r|^p\,d\mu(y)\r\}^{1/p}\le C_5.
\end{equation}
\end{enumerate}
Moreover, the minimal constants $C_3$, $C_4$ and $C_5$ as above are
equivalent to $\|f\|_{\rbmo(\mu)}$.
\end{lem}

\begin{proof}
The equivalence of (i) and (ii) is a special case of
\cite[Proposition 2.2]{hyy10}. Obviously, (iii) implies (ii). By an
argument similar to that used in the proof of \cite[Proposition
2.2]{hyy10}, we have that (ii) implies (iii). Hence, (i), (ii) and
(iii) are equivalent.

We now prove the equivalence of (iii) and (iv). By the H\"older
inequality, it is easy to see that (iv) implies (iii). Conversely,
it follows from \cite[Corollary 6.3]{h10} that for any ball $B$,
\begin{equation*}
\lf\{\frac{1}{\mu(\eta
B)}\int_B\lf|f(y)-f_B\r|^p\,d\mu(y)\r\}^{1/p}\ls\|f\|_{\rbmo(\mu)}.
\end{equation*}
On the other hand, from the equivalence of (i) and (iii), we deduce
that the number $f_B$ in the definition of $\rbmo(\mu)$ can be
chosen to be $m_{{\wz B}^\rho}$. Therefore,
\begin{equation*}
\lf\{\frac{1}{\mu(\eta B)}\int_B\lf|f(y)-m_{{\wz
B}^\rho}(f)\r|^p\,d\mu(y)\r\}^{1/p}\ls\|f\|_{\rbmo(\mu)}\sim\min\{C_4\},
\end{equation*}
which shows that (iii) implies (iv) and hence completes the proof of
Lemma \ref{l3.2}.
\end{proof}

\begin{proof}[Proof of Theorem \ref{t3.1}]
Suppose that $f\in\rbmo(\mu)$ and there exists $x_0\in\cx$ such that
$\cm(f)(x_0)<\fz$. First, we claim that there exists a positive
constant $C$ independent of $f$ such that for all
$(6,\,\bz_6)$-doubling balls $B\ni x_0$,
\begin{equation}\label{e3.7}
\frac{1}{\mu(B)}\int_B\cm(f)(y)\,d\mu(y)\le
C\|f\|_{\rbmo(\mu)}+\inf_{x\in B}\cm(f)(x).
\end{equation}
To prove this, we decompose $f$ as
$$
f=\lf[f-m_B(f)\r]\chi_{3B}+[m_B(f)\chi_{3B}+f\chi_{\cx\setminus(3B)}]\equiv
f_1+f_2.
$$
We choose $\eta\equiv6/5$ and $\rho\equiv6$ in Lemma \ref{l3.2}.
Since $\cm$ is bounded on $L^2(\mu)$, by the H\"older inequality,
\eqref{e3.6}, \eqref{e3.4}, and (ii) and (iii) of Lemma \ref{l2.1},
we have
\begin{align}\label{e3.8}
&\int_B\cm(f_1)(y)\,d\mu(y)\\
&\hs\le[\mu(B)]^{1/2}\lf\{\int_\cx|\cm(f_1)(y)|^2\,d\mu(y)\r\}^{1/2}
\ls[\mu(B)]^{1/2}\lf\{\int_\cx|f_1(y)|^2\,d\mu(y)\r\}^{1/2}\nonumber\\
&\hs\ls[\mu(B)]^{1/2}\lf\{\int_{3B}|f(y)-m_{\wz{3B}^6}(f)|^2\,d\mu(y)\r.\nonumber\\
&\hs\hs\lf.+\int_{3B}|m_B(f)-m_{\wz{3B}^6}(f)|^2\,d\mu(y)\r\}^{1/2}\nonumber\\
&\hs\ls[\mu(B)]^{1/2}\lf\{\lf[\mu\lf(\frac{18}5B\r)\r]^{1/2}
+[\mu(3B)]^{1/2}\lf[1+\dz(B,\,\wz{3B}^6)\r]\r\}\|f\|_{\rbmo(\mu)}\nonumber\\
&\hs\ls\mu(6B)\|f\|_{\rbmo(\mu)}\ls\mu(B)\|f\|_{\rbmo(\mu)}.\nonumber
\end{align}
Next, we show that
\begin{equation}\label{e3.9}
\frac{1}{\mu(B)}\int_B\cm(f_2)(y)\,d\mu(y)\ls\|f\|_{\rbmo(\mu)}+\inf_{x\in
B}\cm(f)(x).
\end{equation}
It suffices to show that for any $y\in B$,
\begin{equation*}
\cm(f_2)(y)\ls\|f\|_{\rbmo(\mu)}+\inf_{x\in B}\cm(f)(x).
\end{equation*}
To this end, it is enough to show that for any
$(6,\,\bz_6)$-doubling ball $S\ni y$ and $y\in B$,
\begin{equation}\label{e3.10}
\frac{1}{\mu(S)}\int_S
f_2(z)\,d\mu(z)\ls\|f\|_{\rbmo(\mu)}+\inf_{x\in B}\cm(f)(x).
\end{equation}
If $S\subset 3B$, we immediately have that
$$
\frac{1}{\mu(S)}\int_S f_2(z)\,d\mu(z)=m_B(f)\le\inf_{x\in
B}\cm(f)(x).
$$
If $S\cap[\cx\setminus(3B)]\neq\emptyset$. Then $r_S>r_B$ and
$3B\subset (5S)$. Write
$$f_2=\lf[m_B(f)-m_{\wz{5S}^6}(f)\r]\chi_{3B}
+\lf[f-m_{\wz{5S}^6}(f)\r]\chi_{\cx\setminus(3B)}+m_{\wz{5S}^6}(f).$$
Obviously, $m_{\wz{5S}^6}(f)\le\inf_{x\in B}\cm(f)(x)$. From
\eqref{e3.5}, it follows that
\begin{align*}
&\int_S\lf\{\lf[m_B(f)-m_{\wz{5S}^6}(f)\r]\chi_{3B}(z)
+\lf[f(z)-m_{\wz{5S}^6}(f)\r]\chi_{\cx\setminus(3B)}(z)\r\}\,d\mu(z)\\
&\hs\le\mu(3B)\lf|m_B(f)-m_{\wz{5S}^6}(f)\r|
+\int_{5S}\lf|f(z)-m_{\wz{5S}^6}(f)\r|\chi_{\cx\setminus(3B)}(z)\,d\mu(z)\\
&\hs\le\frac{\mu(6B)}{\mu(B)}\int_B\lf|f(z)-m_{\wz{5S}^6}(f)\r|\,d\mu(z)
+\int_{5S\setminus(3B)}\lf|f(z)-m_{\wz{5S}^6}(f)\r|\,d\mu(z)\\
&\hs\ls\int_{5S}\lf|f(z)-m_{\wz{5S}^6}(f)\r|\,d\mu(z)
\ls\mu(6S)\|f\|_{\rbmo(\mu)}\ls\mu(S)\|f\|_{\rbmo(\mu)},
\end{align*}
which implies \eqref{e3.10}. Hence, \eqref{e3.9} holds. Combining
the estimates for \eqref{e3.8} and \eqref{e3.9}, we obtain
\eqref{e3.7}.

From \eqref{e3.7}, it follows that for $f\in\rbmo(\mu)$, if
$\cm(f)(x_0)<\fz$ for some point $x_0\in\cx$, then $\cm(f)$ is
$\mu$-finite almost everywhere and in this case,
\begin{equation}\label{e3.11}
\frac{1}{\mu(B)}\int_B\lf[\cm(f)(y)-{\mathop\einf_{x\in
B}}\cm(f)(x)\r]\,d\mu(y)\ls \|f\|_{\rbmo(\mu)},
\end{equation}
provided that $B$ is a $(6,\,\bz_6)$-doubling ball. To prove
$\cm(f)\in\rblo(\mu)$, by Proposition \ref{p2.3}, we still need to
prove that for any $(6,\,\bz_6)$-doubling balls $B\subset S$,
\begin{equation}\label{e3.12}
m_B[\cm(f)]-m_S[\cm(f)]\ls[1+\dz(B,\,S)]\|f\|_{\rbmo(\mu)}.
\end{equation}
To prove \eqref{e3.12}, for any point $x\in B$, we set
$$\cm_1(f)(x)\equiv\sup_{\gfz{P\ni x,\, P\,(6,\,\bz_6)-\rm
doubling}{r_P\le 4r_S}}\frac1{\mu(P)}\int_P f(y)\,d\mu(y),$$
$$\cm_2(f)(x)\equiv\sup_{\gfz{P\ni x,\, P\,(6,\,\bz_6)-\rm
doubling}{r_P> 4r_S}}\frac1{\mu(P)}\int_P f(y)\,d\mu(y),$$ ${\cal
U}_{1,\,B}\equiv\{x\in B:\,\cm_1(f)(x)\ge\cm_2(f)(x)\}$ and ${\cal
U}_{2,\,B}\equiv B\setminus{\cal U}_{1,\,B}$. Then for any $x\in B$,
$\cm(f)(x)=\max[\cm_1(f)(x),\,\cm_2(f)(x)]$. By writing
$$f=[f-m_S(f)]\chi_{3B}+[f-m_S(f)]\chi_{\cx\setminus(3B)}+m_S(f)$$
and using the fact that $m_S(f)\le m_S[\cm(f)]$, we see that
\begin{align*}
m_B[\cm(f)]-m_S[\cm(f)]&\le\frac{1}{\mu(B)}\int_{{\cal U}_{1,\,B}}
\cm_1([f-m_S(f)]\chi_{3B})(x)\,d\mu(x)\\
&\hs+\frac{1}{\mu(B)}\int_{{\cal U}_{1,\,B}}
\cm_1([f-m_S(f)]\chi_{\cx\setminus(3B)})(x)\,d\mu(x)\\
&\hs+\frac{1}{\mu(B)}\int_{{\cal U}_{2,\,B}}
\{\cm_2(f)(x)-m_S[\cm(f)]\}\,d\mu(x)\\
&\equiv{\rm I_1}+{\rm I_2}+{\rm I_3}.
\end{align*}

Notice that $\cm$ is bounded on $L^2(\mu)$. From this, the H\"older
inequality, Lemma \ref{l3.2}, and (ii) and (iii) of Lemma
\ref{l2.1}, it follows that
\begin{align*}
{\rm I_1}&\le\lf\{\frac1{\mu(B)}\int_B
\lf|\cm_1([f-m_S(f)]\chi_{3B})(x)\r|^2\,d\mu(x)\r\}^{1/2}\\
&\ls\lf\{\frac1{\mu(B)}\int_{3B}|f(x)-m_S(f)|^2\,d\mu(x)\r\}^{1/2}\\
&\ls\lf\{\frac1{\mu(B)}\int_{3B}
\lf|f(x)-m_{\wz{3B}^6}(f)\r|^2\,d\mu(x)\r\}^{1/2}
+\lf|m_{\wz{3B}^6}(f)-m_B(f)\r|\\
&\hs+|m_B(f)-m_S(f)|\\
&\ls[1+\dz(B,\,S)]\|f\|_{\rbmo(\mu)}.
\end{align*}

To estimate ${\rm I_2}$, we first claim that for any point $x\in B$
and any $(6,\,\bz_6)$-doubling ball $P\ni x$ with $r_P\le 4r_S$,
\begin{align}\label{e3.13}
{\rm J}&\equiv\frac1{\mu(P)}\int_P|f(y)-m_S(f)|
\chi_{\cx\setminus(3B)}(y)\,d\mu(y)\\
&\ls[1+\dz(B,\,S)]\|f\|_{\rbmo(\mu)}.\nonumber
\end{align}
If $P\subset 3B$, then ${\rm J}=0$ and \eqref{e3.13} holds
automatically. Assume that $P\not\subset 3B$. We then have that
$r_P>r_B$, which together with the fact that $r_P\le 4r_S$ implies
that $B\subset 3P\subset 17S$. Thus, \eqref{e3.3} and \eqref{e3.4},
together with (ii), (iii) and (iv) of Lemma \ref{l2.1}, yield that
\begin{align*}
{\rm J}&\le\frac1{\mu(P)}\int_P|f(y)-m_P(f)|\,d\mu(y)
+\lf|m_P(f)-m_{\wz{3P}^6}(f)\r|\\
&\hs+\lf|m_{\wz{3P}^6}(f)-m_B(f)\r|+|m_B-m_S(f)|
\ls[1+\dz(B,\,S)]\|f\|_{\rbmo(\mu)},
\end{align*}
which further implies that for all $x\in B$,
$$
\cm_1([f-m_S(f)]\chi_{\cx\setminus(3B)})(x)\ls[1+\dz(B,\,S)]\|f\|_{\rbmo(\mu)}.
$$
From this, we deduce that ${\rm
I_2}\ls[1+\dz(B,\,S)]\|f\|_{\rbmo(\mu)}$.

Now we estimate ${\rm I_3}$. Notice that for any $x\in B$, any
$(6,\,\bz_6)$-doubling ball $P$ containing $x$ with $r_P>4r_S$ and
$B\subset S$, $S\subset3P$. Then from \eqref{e3.4} and the fact
$m_{\wz{3P}^6}(f)\le m_S[\cm(f)]$, it follows that
\begin{align*}
m_P(f)-m_S[\cm(f)]&\le\lf|m_P(f)-m_{\wz{3P}^6}(f)\r|
+m_{\wz{3P}^6}(f)-m_S[\cm(f)]\\
&\ls\|f\|_{\rbmo(\mu)}.
\end{align*}
Taking the supremum over all $(6,\,\bz_6)$-doubling balls $P$
containing $x$ with $r_P>4r_S$, we have that for all $x\in B$,
$$
\cm_2(f)(x)-m_S[\cm(f)]\ls\|f\|_{\rbmo(\mu)}.
$$
This implies that ${\rm I_3}\ls\|f\|_{\rbmo(\mu)}$.

Combining the estimates for ${\rm I_1}$ through ${\rm I_3}$ leads to
\eqref{e3.12}, which together with \eqref{e3.11} implies that $\cm$
is bounded from $\rbmo(\mu)$ to $\rblo(\mu)$ and hence completes the
proof of Theorem \ref{t3.1}.
\end{proof}

\section{Boundedness of maximal Calder\'on-Zygmund
operators}\label{s4}

\hskip\parindent This section is devoted to the boundedness of the
maximal operators associated with the Calder\'on-Zygmund operators
introduced in \cite{hm}.

Let $\triangle\equiv\{(x,\,x):\,x\in\cx\}$ and $L^\fz_b(\cx)$ denote
the {\it space of all functions in $L^\fz(\cx)$ with bounded
support}. A {\it standard kernel} is a mapping
$K:\,(\cx\times\cx)\backslash\triangle\rightarrow\mathbb{C}$ for
which, there exist some positive constants $\sz$ and $C$ such that
for all $x,\,y\in\cx$ with $x\neq y$,
\begin{equation}\label{e4.1}
|K(x,\,y)|\le C\frac{1}{\lz(x,\,d(x,\,y))},
\end{equation}
and that for all $x,\,{\wz x},\,y\in\cx$ with $d(x,\,{\wz x})\le
\frac{d(x,\,y)}2$,
\begin{equation}\label{e4.2}
|K(x,\,y)-K({\wz x},\,y)|+|K(y,\,x)-K(y,\,{\wz x})|\le
C\frac{[d(x,\,{\wz x})]^\sz}{[d(x,\,y)]^\sz\lz(x,\,d(x,\,y))}.
\end{equation}
A linear operator $T$ is called a {\it Calder\'on-Zygmund operator}
with kernel $K$ satisfying \eqref{e4.1} and \eqref{e4.2} if for all
$f\in L^{\infty}_b(\cx)$ and $x\not\in\supp(f)$,
\begin{equation}\label{e4.3}
Tf(x)\equiv\int_\cx K(x,\,y)f(y)\,d\mu(y).
\end{equation}

A new example of operators with kernel satisfying \eqref{e4.1} and
\eqref{e4.2} is the so-called Bergman-type operator appearing in
\cite{vw}; see also \cite{hm} for an explanation.

Now, we define the corresponding maximal Calder\'on-Zygmund operator
associated with the kernel $K$. For any $\ez\in(0,\,\fz)$, define
the {\it truncated operator} $T_\ez$ by setting, for all $x\in\cx$,
\begin{equation}\label{e4.4} T_\ez
f(x)\equiv\int_{d(x,\,y)>\ez} K(x,\,y)f(y)\,d\mu(y).
\end{equation}
The {\it maximal Calder\'on-Zygmund operator} $T_*$ is defined by
setting, for all $x\in\cx$,
\begin{equation}\label{e4.5} T_*
f(x)\equiv\sup_{\ez>0}|T_\ez f(x)|.
\end{equation}

\begin{rem}\label{r4.1}\rm
Let $\cx\equiv\rn$. It is well known that if $\mu$ is the
$n$-dimensional Lebesgue measure and $T$ bounded on $L^2(\rn)$, then
$T_*$ is bounded from $L^\fz(\mu)$ to $\mathop\mathrm{BMO}(\rn)$
(see \cite{s67}), and furthermore, bounded from $L^\fz(\mu)$ to
$\mathop\mathrm{BLO}(\rn)$ (see \cite{l85}). When $\mu$ is a
non-doubling measure with polynomial growth, Tolsa \cite{t01} proved
that if $T$ is bounded on $L^2(\mu)$, then $T$ is bounded from
$L^\fz(\mu)$ to $\rbmo(\mu)$, and moreover, the boundedness of $T_*$
from $L^\fz(\mu)$ to $\rblo(\mu)$ was obtained by Jiang \cite{j05}.
\end{rem}

It was proved in \cite{hlyy} that if the Calder\'on-Zygmund operator
$T$ is bounded on $L^2(\mu)$, then the maximal operator $T_*$ is of
weak type $(1,\,1)$ and bounded on $L^p(\mu)$ for any
$p\in(1,\,\fz)$. On the boundedness of $T_*$ when $p=\fz$, we have
the following conclusion.

\begin{thm}\label{t4.1}
Let $T$ be the Calder\'on-Zygmund operator as in \eqref{e4.3} with
kernel $K$ satisfying \eqref{e4.1} and \eqref{e4.2}. If $T$ is
bounded on $L^2(\mu)$, then the maximal operator $T_*$ as in
\eqref{e4.5} is bounded from $L^\fz(\mu)$ to $\rblo(\mu)$.
\end{thm}

\begin{proof}
First we claim that there exists a positive constant $C$ such that
for all $f\in L^{\infty}(\mu)\cap L^{p_0}(\mu)$, $p_0\in[1,\,\fz)$,
and $(6,\,\bz_6)$-doubling balls $B$,
\begin{equation}\label{e4.6}
\frac{1}{\mu(B)}\int_B T_*f(x)\,d\mu(x)\le
C\|f\|_{L^{\infty}(\mu)}+\inf_{y\in B}T_*f(y).
\end{equation}

To prove this, we decompose $f$ as
$$f=f\chi_{5B}+f\chi_{\cx\setminus(5B)}\equiv f_1+f_2.$$
By the H\"older inequality and the $L^2(\mu)$-boundedness of $T_*$,
we have
\begin{align}\label{e4.7}
\frac1{\mu(B)}\int_B T_*f_1(x)\,d\mu(x) &\le\frac1{[\mu(B)]^{1/2}}
\lf\{\int_\cx[T_*(f\chi_{5B})(x)]^2\,d\mu(x)\r\}^{1/2}\\
&\ls\frac1{[\mu(B)]^{1/2}}
\lf\{\int_\cx|f\chi_{5B}(x)|^2\,d\mu(x)\r\}^{1/2}\nonumber\\
&\ls\frac{[\mu(5B)]^{1/2}}{[\mu(B)]^{1/2}}\|f\|_{L^\fz(\mu)}
\ls\|f\|_{L^\fz(\mu)}.\nonumber
\end{align}
From \eqref{e1.3} and \eqref{e1.2}, we deduce that for any ball $B$,
$y\not\in 5B$ and $x\in B$,
\begin{equation}\label{e4.8}
\lz(c_B,\,d(y,\,c_B))\sim
\lz(y,\,d(y,\,c_B))\sim\lz(y,\,d(y,\,x))\sim\lz(x,\,d(y,\, x)).
\end{equation}
Notice that
$$
\{y\in\cx:\,d(x,\,y)>6r_B\,\,{\rm for\,\,some}\,\,x\in
B\}\subset[\cx\setminus (5B)].
$$
It then follows from \eqref{e4.1}, \eqref{e4.8} and Lemma
\ref{l2.1}(ii) that for all $y\in B$,
\begin{align}\label{e4.9}
T_*f_2(y)&\le\max\lf\{\sup_{\ez\ge 6r_B}\lf|T_\ez
f_2(y)\r|,\,\sup_{0<\ez<6r_B}\lf|T_\ez f_2(y)\r|\r\}\\
&\le\max\lf\{T_*f(y),\,\sup_{0<\ez<6r_B}\lf|\int_{d(y,\,z)>6r_B}
K(y,\,z)f_2(z)\,d\mu(z)\r.\r.\nonumber\\
&\hs\lf.\lf.+\int_{\ez<d(y,\,z)\le6r_B}
K(y,\,z)f_2(z)\,d\mu(z)\r|\r\}\nonumber\\
&\le
T_*f(y)+C\|f\|_{L^\fz(\mu)}\sup_{0<\ez<6r_B}\int_{(7B)\setminus(5B)}
\frac1{\lz(y,\,d(y,\,z))}\,d\mu(z)\nonumber\\
&\le T_*f(y)+C\|f\|_{L^\fz(\mu)}\int_{(8B)\setminus B}
\frac1{\lz(c_B,\,d(z,\,c_B))}\,d\mu(z)\nonumber\\
&=T_*f(y)+C\|f\|_{L^\fz(\mu)}\dz(B,\,4B)\le
T_*f(y)+C\|f\|_{L^\fz(\mu)},\nonumber
\end{align}
where $C$ is a positive constant independent of $f$ and $y$. Thus,
the proof of the estimate \eqref{e4.6} is reduced to proving that
for all $x,\,y\in B$,
\begin{equation}\label{e4.10}
|T_*f_2(x)-T_*f_2(y)|\ls \|f\|_{L^{\infty}(\mu)}.
\end{equation}
To this end, for any $\ez\in(0,\,\fz)$, write
\begin{align*}
&|T_\ez f_2(x)-T_\ez f_2(y)|\\
&\hs=\lf|\int_{d(x,\,z)>\ez}
K(x,\,z)f_2(z)\,d\mu(z)-\int_{d(y,\,z)>\ez}
K(y,\,z)f_2(z)\,d\mu(z)\r|\\
&\hs\le\int_{\gfz{d(x,\,z)>\ez}{d(y,\,z)>\ez}}|
K(x,\,z)-K(y,\,z)||f_2(z)|\,d\mu(z)\\
&\hs\hs+\int_{\gfz{d(x,\,z)>\ez}{d(y,\,z)\le\ez}}|K(x,\,z)f_2(z)|\,d\mu(z)\\
&\hs\hs+\int_{\gfz{d(y,\,z)>\ez}{d(x,\,z)\le\ez}}|K(y,\,z)f_2(z)|\,d\mu(z)
\equiv{\rm J_1}+{\rm J_2}+{\rm J_3}.
\end{align*}

By \eqref{e4.2}, \eqref{e4.8} and \eqref{e1.2}, we have that for all
$x,\,y\in B$,
\begin{align*}
{\rm J_1}
&\le\int_{\cx\setminus (5B)}|K(x,\,z)-K(y,\,z)||f(z)|\,d\mu(z)\\
&\ls\|f\|_{L^\fz(\mu)}\int_{\cx\setminus
(5B)}\frac{[d(x,\,y)]^\sz}{[d(x,\,z)]^\sz\lz(x,\,d(x,\,z))}\,d\mu(z)\\
&\ls\|f\|_{L^\fz(\mu)}\int_{\cx\setminus
(5B)}\lf[\frac{r_B}{d(z,\,c_B)}\r]^\sz
\frac1{\lz(c_B,\,d(z,\,c_B))}\,d\mu(z) \ls\|f\|_{L^\fz(\mu)}.
\end{align*}

Now we estimate ${\rm J_2}$. Notice that if $z\not\in 5B$ and $x\in
B$, then $d(x,\,z)>4r_B$. Therefore, for any $\ez\in(0,\,4r_B]$ and
$x,\,y\in B$, $\{z\not\in 5B:\,d(x,\,z)>\ez\,\,{\rm and}\,\,
d(y,\,z)\le\ez\}=\emptyset$. So, we only need to consider the case
that $\ez\in(4r_B,\,\fz)$. In this case, there exists a unique
$m\in\nn$ such that $2^{m-1}r_B<\ez\le2^m r_B$, which leads to that
$$
\{z\not\in 5B:\,d(x,\,z)>\ez\,\,{\rm and}\,\,
d(y,\,z)\le\ez\}\subset [2^{m+1}B\setminus(\max(2,\,2^{m-1}-1)B)].
$$
This, together with \eqref{e4.1}, \eqref{e4.8} and  Lemma
\ref{l2.1}(ii), shows that
\begin{align*}
{\rm
J_2}&\ls\|f\|_{L^\fz(\mu)}\int_{2^{m+1}B\setminus(\max(2,\,2^{m-1}-1)B)}
\frac1{\lz(c_B,\,d(z,\,c_B))}\,d\mu(z)\\
&\ls\|f\|_{L^\fz(\mu)}\dz(\max(2,\,2^{m-1}-1)B,\,2^m
B)\ls\|f\|_{L^\fz(\mu)}.
\end{align*}

An argument similar to the estimate of ${\rm J_2}$ also yields that
${\rm J_3}\ls\|f\|_{L^\fz(\mu)}$. Combining the estimates for ${\rm
J_1}$ through ${\rm J_3}$ implies \eqref{e4.10} and hence
\eqref{e4.6} holds.

Thus, by \eqref{e4.6}, we know that if $f\in L^{\infty}(\mu)\cap
L^{p_0}(\mu)$ with $p_0\in[1,\,\fz)$, then $T_*f$ is $\mu$-finite
almost everywhere and in this case, by \eqref{e4.6} again, we have
that
\begin{equation*}
\frac{1}{\mu(B)}\int_B\lf[T_*f(x)-{\mathop\einf_{y\in B}}T_*
f(y)\r]\,d\mu(x)\ls \|f\|_{L^\fz(\mu)},
\end{equation*}
provided that $B$ is a $(6,\,\bz_6)$-doubling ball. To prove
$T_*f\in\rblo(\mu)$, by Proposition \ref{p2.3}, we still need to
prove that $T_*f$ satisfies \eqref{e2.9}. Let $B\subset S$ be any
two $(6,\,\bz_6)$-doubling balls. For any $\ez\in(0,\,\fz)$, $x\in
B$ and $y\in S$, we set
\begin{align*}
T_\ez f(x)&=T_\ez
(f\chi_{5B})(x)+T_\ez(f\chi_{(5S)\setminus(5B)})(x)\\
&\hs+\lf[T_\ez(f\chi_{\cx\setminus(5S)})(x)-T_\ez(f\chi_{\cx\setminus(5S)})(y)\r]
+T_\ez(f\chi_{\cx\setminus(5S)})(y).
\end{align*}
By an estimate similar to that of \eqref{e4.9}, we have that for all
$y\in S$,
$$
T_*(f\chi_{\cx\setminus(5S)})(y)\le T_*f(y)+C\|f\|_{L^\fz(\mu)},
$$
where $C$ is a positive constant independent of $f$ and $y$. On the
other hand, by the estimate same as that of \eqref{e4.10}, we have
that for all $x,\, y \in S$,
$$
\lf|T_\ez(f\chi_{\cx\setminus(5S)})(x)
-T_\ez(f\chi_{\cx\setminus(5S)})(y)\r|\ls\|f\|_{L^{\infty}(\mu)}.
$$
For all $x\in B$, if $z\not\in 5B$, then $d(x,\,z)\ge 4r_B$, which,
together with \eqref{e4.4}, \eqref{e4.1} and \eqref{e4.8}, shows
that
\begin{align*}
T_\ez(f\chi_{(5S)\setminus(5B)})(x)&=\int_{d(x,\,z)>\ez}
K(x,\,z)f\chi_{(5S)\setminus(5B)}(z)\,d\mu(z)\\
&\ls\|f\|_{L^\fz(\mu)}\int_{(5S)\setminus(5B)}|K(x,\,z)|\,d\mu(z)\\
&\ls\|f\|_{L^\fz(\mu)}\int_{(5S)\setminus(5B)}\frac{1}{\lz(x,\,d(x,\,z))}\,d\mu(z)\\
&\ls\|f\|_{L^\fz(\mu)}\int_{(5S)\setminus
B}\frac{1}{\lz(c_B,\,d(z,\,c_B))}\,d\mu(z)\\
&\ls[1+\dz(B,\,S)]\|f\|_{L^\fz(\mu)}.
\end{align*}
Thus,
$$
T_*f(x)\ls
T_*(f\chi_{5B})(x)+[1+\dz(B,\,S)]\|f\|_{L^\fz(\mu)}+T_*f(y).
$$
Taking mean value over $B$ for $x$, and over $S$ for $y$, we then
obtain
$$
m_B(T_*f)-m_S(T_*f)\ls[1+\dz(B,\,S)]\|f\|_{L^\fz(\mu)},
$$ where we used \eqref{e4.7}. This
finishes the proof of Theorem \ref{t4.1} in the case of $f\in
L^{\infty}(\mu)\cap L^{p_0}(\mu)$ with $p_0\in[1,\,\fz)$.

If $f\in L^\fz(\mu)$ and $f\not\in L^p(\mu)$ for all
$p\in[1,\,\fz)$, then the integral
$$\int_{d(x,\,y)>\ez}K(x,\,y)f(y)\,d\mu(y)$$
may not be convergent. The operator $T_\ez$ can be extended to the
whole space $L^\fz(\mu)$ by following the standard arguments (see,
for example, \cite[p.\,105]{t01}): Fix any point $x_0\in\cx$. For
any given ball $B(x_0, r)$ centered at $x_0\in\cx$ with the radius
$r>3\ez$, we write $f=f_1+f_2$, with $f_1\equiv f\chi_{B(x_0,3r)}$.
For $x\in B(x_0,r)$, we then define
$$
T_\ez f(x)=T_\ez
f_1(x)+\int_{d(x,\,y)>\ez}[K(x,\,y)-K(x_0,\,y)]f_2(y)\,d\mu(y).
$$
Now both integrals in this equation are convergent. Using this
definition, Remark \ref{r2.1}(i) and then repeating the argument as
above then completes the proof of Theorem \ref{t4.1}.
\end{proof}

\bigskip

\noindent Haibo Lin

\medskip

\noindent College of Science, China Agricultural University, Beijing
100083, People's Republic of China

\medskip

\noindent{\it E-mail address}: \texttt{haibolincau@126.com}

\bigskip

\noindent Dachun Yang (Corresponding author)

\medskip

\noindent School of Mathematical Sciences, Beijing Normal
University, Laboratory of Mathematics and Complex systems, Ministry
of Education, Beijing 100875, People's Republic of China

\medskip

\noindent{\it E-mail address}: \texttt{dcyang@bnu.edu.cn}

\end{document}